\documentclass[11pt]{amsart}
\usepackage{geometry}               
\usepackage{amssymb}
\usepackage{epstopdf}
\DeclareMathOperator{\inv}{\rm inv}
\newtheorem{thm}{Theorem}[section]
\newtheorem{cor}[thm]{Corollary}
\newtheorem{prop}[thm]{Proposition}

\theoremstyle{definition}
\newtheorem{ex}[thm]{Example}
\newtheorem{rem}[thm]{Remark}
\newtheorem{defi}[thm]{Definition}
\title{Non-Nudgable subgroups of permutations}
\author{Tim Netzer}
 \thanks{The author thanks Nick Netzer for bringing up the main question of this article, and both Steffen Kionke and Nick Netzer for inspiring discussions on the topic.}                                     

\begin{document}
\maketitle

\begin{abstract}
Motivated by a problem from behavioral economics, we study subgroups of permutation groups that have a certain strong symmetry. Given a fixed permutation, consider the set of all permutations with disjoint inversion sets. The group is called {\it non-nudgable}, if the cardinality of this set always remains the same when replacing the initial permutation with its inverse. It is called {\it nudgable } otherwise. We show that all full permutation groups, standard dihedral groups, half of the alternating groups, and any abelian subgroup are non-nudgable. In the right probabilistic sense, it is thus quite likely that a randomly generated subgroup is non-nudgable. However, the other half of the alternating groups are nudgable. We also construct a smallest possible nudgable group, a $6$-element subgroup of the permutation group on $4$ elements.
\end{abstract}

\section{Introduction}
Let $\mathcal X$ be a finite set, whose elements represent certain alternatives that an individual may choose from. A total linear order on $\mathcal X$ then represents the preferences of an individual as to the given alternatives. Let $\mathcal O(\mathcal X)$ be the set of all total orders on $\mathcal X$. Certain decisions processes from behavioral economics are modeled in this setup, see for example \cite{ne,ru}. Assume the real preferences of an individual are not known, only a subset $G\subseteq\mathcal O(\mathcal X)$ of possible preferences is revealed. Certain mechanisms may prompt the individual to behave like it has a certain preference from this set (although it is maybe not the real preference). Such methods are often subsumed under the notion of {\it nudging} \cite{th}. When designing such a mechanism, it might be interesting to see how far from the real preference  it makes the individual deviate.  Or, when deciding between two possible mechanisms, one would like to see which one will make the individual deviate less from his real preference.

So we would like to compare certain orders to each other, and in particular ask whether one order is closer to a  first than to a second one. For this let $\pi_1,\pi_2,\pi\in G$ and say that {\it $\pi$ is closer to $\pi_1$ than to $\pi_2$}, if whenever $\pi_1$ and $\pi_2$ order two alternatives differently, then $\pi$ orders them just as $\pi_1$ does. As a formula, where orders are understood as binary relations on $\mathcal X$, this means $$\pi_1\setminus\pi_2\ \subseteq\ \pi. $$ In this way we obtain two sets, namely the set $C_1$ of orders from $G$ that are closer to $\pi_1$ than to $\pi_2$, and the set $C_2$ of those orders that are closer to $\pi_2$ than to $\pi_1.$ Note that not every order necessarily belongs to either $C_1$ or $C_2$. The size of $C_1$ compared to the size of $C_2$ shows the probability that $\pi_1$ will inforce less deviations from the real preference than $\pi_2$. 

It has been observed in \cite{ne} that in many cases the two sets $C_1$ and $C_2$ are of the same cardinality,  meaning that no mechanism is predominant in pairwise comparison. This happens for many of the relevant models from the literature, see for example \cite{ne} and the references therein. Since the question is well-motivated  from the economical applications, it clearly asks for a systematic mathematical treatment. It can be formulated in terms of permutation groups and inversions, which we will explain in the following. Our main section then proves several results on whether subgroups of permutation groups fulfill this property of {\it non-nudgability}. Among them are all full permutation groups, half of the alternating groups, all standard dihedral groups and all abelian subgroups. But there are also groups which are {\it nudgable}, for example the other half of all alternating groups. The smallest example of a nudgable group is a 6-element subgroup of $\mathcal S_4$. In a suitable formulation, a randomly generated subgroup of $\mathcal S_n$ is more likely to be non-nudgable than nudgable, if $n$ is large.

\section{Preliminaries}

Let us first translate the initial problem from the introduction into a question about subgroups of permutations groups. For a general background on permutation groups we recommend \cite{di,ro} or any other introductory text to algebra.  Assume without loss of generality $\mathcal X=\{1,\ldots, n\}.$ Then  the set $\mathcal O(\mathcal X)$ of total orders on $\mathcal X$ can be identified with the permutation group $\mathcal S_n$; we will identify the total order   $i_1 < i_2 < \ldots < i_n$ with the permutation $$\pi=\left(\begin{array}{cccc}1 & 2 & \cdots & n \\i_1 & i_2 & \cdots & i_n\end{array}\right),$$ i.e. $\pi$ maps $j$ to $i_j$. Any subset $G\subseteq\mathcal O(\mathcal X)$ then becomes a subset $G\subseteq \mathcal S_n$. In most of the applications, the set $G$ will even be a subgroup (see for example \cite{ne} and the references therein). Since this assumption is natural from a mathematical point of view as well, we will restrict to subgroups from now in.  Now let $\pi_1,\pi_2,\pi\in\mathcal O(\mathcal X)$ and consider the condition $\pi_1\setminus\pi_2\subseteq \pi.$ It means $$i<_{\pi_1}j\quad \wedge\quad   j<_{\pi_2} i \quad \Rightarrow \quad i<_\pi j$$ or, in the language of permutations, \begin{equation*}\label{per}\pi_1^{-1}(i)<\pi_1^{-1}(j) \quad\wedge\quad \pi_2^{-1}(j)<\pi_2^{-1}(i)\quad \Rightarrow\quad \pi^{-1}(i)<\pi^{-1}(j).\end{equation*} By setting $k=\pi_1^{-1}(i)$ and $l:=\pi_1^{-1}(j)$ this becomes \begin{equation}\label{per2}  k<l \quad\wedge\quad \pi_2^{-1}(\pi_1(l)) < \pi_2^{-1}(\pi_1(k)) \quad\Rightarrow\quad \pi^{-1}(\pi_1(k))<\pi^{-1}(\pi_1(l)) .\end{equation}
Now recall the notion of an {\it inversion} of a permutation $\sigma$. It is a pair whose order is   reversed by $\sigma$: $${\rm inv}(\sigma):=\left\{ (i,j)\mid 1\leq i<j\leq n, \sigma(j)<\sigma(i)\right\}.$$ Now equation (\ref{per2}) becomes $$\inv(\pi_2^{-1}\pi_1)\cap \inv(\pi^{-1}\pi_1)=\emptyset.$$
With this formulation we define the sets $C_1,C_2$  from above as follows: $$C_1=\left\{ \pi\in G\mid \inv(\pi_2^{-1}\pi_1)\cap \inv(\pi^{-1}\pi_1)=\emptyset\right\}$$ and 
$$C_2=\left\{ \pi\in G\mid \inv(\pi_1^{-1}\pi_2)\cap \inv(\pi^{-1}\pi_2)=\emptyset\right\}.$$ Now finally note that $\pi_1^{-1}\pi_2$ is the inverse of $\pi_2^{-1}\pi_1$, and with $\pi$ running through the full  subgroup $G$, so do $\pi^{-1}\pi_1$ and $\pi^{-1}\pi_2.$ So we set 
$$ D_G(\pi):=\left\{ \sigma\in G\mid {\inv}(\sigma)\cap \inv(\pi)=\emptyset\right\}$$ and define:

\begin{defi} A  subgroup $G\subseteq \mathcal S_n$ is called  {\it non-nudgable} if for all $\pi\in G$ we have
$$\vert D_G(\pi)\vert =\vert D_G(\pi^{-1})\vert.$$ Otherwise $G$ is called {\it nudgable}.
\end{defi}

\section{Main Results}

Our first result is a straightforward observation, but will already cover a large class of groups. For this let $$\omega_0=\left(\begin{array}{cccc}1 & 2 & \cdots & n \\n & n-1 & \cdots  & 1\end{array}\right)\in\mathcal S_n$$ denote the permutation that has all ordered pairs as inversions.

\begin{thm}\label{thm1}
If $G\subseteq \mathcal S_n$ is a subgroup that contains $\omega_0$, then $G$ is non-nudgable.
\end{thm}
\begin{proof} We prove that for $\sigma\in D_G(\pi)$ we have $\omega_0\sigma\pi^{-1}\in D_G(\pi^{-1})$. Then clearly
 the mapping \begin{align*}D_G(\pi)&\to D_G(\pi^{-1})\\ \sigma&\mapsto \omega_0\sigma\pi^{-1}\end{align*} is a bijection. Since multiplying with $\omega_0$ from the left just exchanges inversions and non-inversions, it remains to prove that $$\inv(\pi^{-1})\subseteq \inv(\sigma\pi^{-1}).$$ But this is clear, since $(i,j)\in\inv(\pi^{-1})$ just means $(\pi^{-1}(j),\pi^{-1}(i))\in\inv(\pi)$, and thus $(\pi^{-1}(j),\pi^{-1}(i))\notin\inv(\sigma)$.\end{proof}

\begin{cor}\label{cor1} The following subgroups of $\mathcal S_n$ are non-nudgable:
\begin{itemize}
\item[(i)] The full group $\mathcal S_n$.
\item[(ii)] The alternating group $\mathcal A_n$ if $\lfloor \frac{n}{2} \rfloor $ is even.
\item[(iii)] The dihedral group $\mathcal D_n,$ generated by $\omega_0$ and the cycle $(12\cdots n)$.
\end{itemize}
\end{cor}

\begin{prop}\label{an}
If $n\geq 6$ and $\lfloor\frac{n}{2}\rfloor$ is odd, then $\mathcal A_n$ is nudgable.
\end{prop}
\begin{proof}
Consider the permutation $$\pi=\left(\begin{array}{cccccccccc}1 & 2 & 3 & \cdots &n-5& n-4 & n-3 & n-2 & n-1 & n \\n & n-1 & n-2 & \cdots &6& 5 & 3 & 1 & 4 & 2\end{array}\right).$$ The only  ordered pairs that are not inversions of $\pi$ are $$(n-3,n-1),\ (n-2,n-1),\ (n-2,n).$$ From this we see that $\pi$ belongs to $\mathcal A_n$ precisely in the cases considered here. It is now also not hard to explicitly compute 
$$D_{\mathcal A_n}(\pi)=\left\{{\rm id},(n-2,n,n-1),(n-3,n-2,n-1) \right\}.$$ On the other hand, the only ordered pairs that are not inversions of $\pi^{-1}$ are $$(1,2),\ (1,4),\ (3,4).$$ From this we compute $$D_{\mathcal A_n}(\pi^{-1})=\left\{ {\rm id}, (12)(34)\right\}.$$ So $\vert D_{\mathcal A_n}(\pi)\vert =3 > 2=\vert D_{\mathcal A_n}(\pi^{-1})\vert.$
\end{proof}

%

For our next result we will need the following straightforward observations:  $$\inv(\pi^{-1})=\pi\inv(\pi)$$  $$\inv(\sigma\pi)=\inv(\pi)\ \triangle\  \pi^{-1}\inv(\sigma)$$ where $\triangle$ denotes the symmetric difference. Here we use the notation $\sigma\inv(\pi)$ for  the set we obtain by applying $\sigma$ to both entries of each $(i,j)$ from $\inv(\pi)$, and permute (if necessary) to make the first entry smaller than the second one.

\begin{thm}\label{abel}
Any abelian subgroup $G\subseteq \mathcal S_n$ is non-nudgable. \end{thm}
\begin{proof}
We claim that $\sigma\in D_G(\pi)$ implies $\sigma^{-1}\in D_G(\pi^{-1})$, which will prove the claim.
By definition, $\sigma\in D_G(\pi)$ means $\inv(\sigma)\cap\inv(\pi)=\emptyset,$ which implies $$\emptyset=\sigma \inv(\sigma)\cap \sigma \inv(\pi)=\inv(\sigma^{-1})\cap\sigma\inv(\pi).$$ From this we obtain $$\inv(\pi\sigma^{-1})=\inv(\sigma^{-1})\ \triangle\ \sigma\inv(\pi)=\inv(\sigma^{-1}) \ \dot\cup \ \sigma\inv(\pi).$$On the other hand, by commutativity, we obtain $$\inv(\pi\sigma^{-1})=\inv(\sigma^{-1}\pi)=\inv(\pi)\ \triangle\ \pi^{-1}\inv(\sigma^{-1}).$$ Now note that $\inv(\pi)$ and $\sigma\inv(\pi)$ have the same cardinality, and similar with $\pi^{-1}\inv(\sigma^{-1})$ and $\inv(\sigma^{-1})$. We conclude $\inv(\pi)\cap \pi^{-1}\inv(\sigma^{-1})=\emptyset,$ and so $$\emptyset=\pi\inv(\pi)\cap\inv(\sigma^{-1})=\inv(\pi^{-1})\cap\inv(\sigma^{-1}),$$ which means $\sigma^{-1}\in D_G(\pi^{-1})$.
\end{proof}

\begin{ex}
The groups $\mathcal A_2,\mathcal A_3$ are non-nudgable, since they are abelian. Thus we have fully classified nudgability of alternating groups.
\end{ex}

Before we give a list of more examples, we note how new non-nudgable subgroups can be constructed from known ones. The case of a product  $\mathcal S_{n_1}\times \cdots\times  \mathcal S_{n_r}$ is one of the important cases in \cite{ne}.

\begin{prop}\label{prod}
For  $n=n_1+\cdots +n_r$ decompose $$\{1,\ldots,n\}=\{1,\ldots, n_1\}\cup\{ n_1+1,\ldots, n_1+n_2\} \cup \cdots $$ and embed $\mathcal S_{n_1}\times \cdots\times  \mathcal S_{n_r}$ into $\mathcal S_{n}$ by letting $\mathcal S_{n_i}$ permute the numbers in the $i$-th subset.  If $G_i\subseteq \mathcal S_{n_i}$ are non-nudgable subgroups,  then so is  $$G_1\times\cdots\times G_r\subseteq \mathcal S_n.$$
\end{prop}
\begin{proof}This is clear, since for $\pi=(\pi_1,\ldots,\pi_r)\in G:=G_1\times\cdots\times G_r$ we have $\pi^{-1}=(\pi_1^{-1},\ldots,\pi_r^{-1})$ and $$D_G(\pi)=D_{G_1}(\pi_1) \times \cdots \times D_{G_r}(\pi_r). $$
\end{proof}

\begin{rem}\label{rem1}
In Theorem \ref{thm1}, the condition $\omega_0\in G$ can clearly be weakened to the existence of some permutation $\tau\in G$, whose inversion set contains all other inversions sets of group elements. However, we where not able to produce an example of such a group, which does not arise as a product (as in Proposition \ref{prod}) of groups to which Theorem \ref{thm1} applies directly.
However, one can even weaken the condition to the following one: \begin{equation}\label{eq}\forall \pi\in G \ \exists \tau\in G\ \forall \sigma\in D_G(\pi)\colon\quad \sigma \inv(\pi)\subseteq\inv(\tau).\end{equation} In this case, a bijection from $D_G(\pi)$ to $D_G(\pi^{-1})$ is given by $$\sigma\mapsto \tau\sigma\pi^{-1},$$ as is proven similar to the proof of Theorem \ref{thm1}. One example of such a group, which is not covered by any of the other results, is given in Example \ref{ex} (iv) below.
\end{rem}

Now let us give more explicit examples:

\begin{ex}\label{ex}
(i) Any subgroup of $\mathcal S_1,\mathcal S_2$ and $\mathcal S_3$ is non-nudgable. In fact all nontrivial subgroups are abelian, so Corollary \ref{cor1} and Theorem \ref{abel} apply.

 (ii) Proposition \ref{prod} does not hold for ''non-diagonal'' embeddings. For $n\geq 4$ we embedd $\mathcal S_{n-1}$ into $\mathcal S_{n}$ by fixing $2$. This leads to a nudgable subgroup $G$. For the element $$\pi=\left(\begin{array}{ccccccccc}1 & 2 & 3 & 4 & \cdots & n-3 & n-2 & n-1 & n \\n & 2 & n-1 & n-2 & \cdots & 5 & 4 & 1 & 3\end{array}\right)$$ one computes, in a similar fashion as in Proposition \ref{an}: $$\vert D_{G}(\pi)\vert=2 > 1=\vert D_{G}(\pi^{-1})\vert.$$
For $n=4$  we obtain a $6$-element subgroup of $\mathcal S_4$ that is nudgable. Since all smaller subgroups are abelian, this is the smallest possible example of a nudgable subgroup.


(iii) The group $G\subseteq\mathcal S_5$ generated by $(12)(34)$ and $(15)(23)$ hat $10$ elements, is non-abelian, and does not contain an element with largest inversion set. So none of the above main results apply. However, condition (\ref{eq}) from Remark \ref{rem1} is fulfilled, and $G$ is thus non-nudgable. This can easily be verifed.
\end{ex}

\begin{rem}
If a subgroup $G\subseteq\mathcal S_n$ is generated by one permutation, it is abelian and thus non-nudgable by Theorem \ref{abel}. If two permutations are chosen independently uniformly at random, the probability that they generate $\mathcal S_n$ goes to $3/4$ for $n\to \infty$.  The probability that either $\mathcal S_n$ or $\mathcal A_n$ is generated even goes to $1$. These are  the main results of \cite{dix}. So the probability that a randomly generated subgroup of $\mathcal S_n$ is non-nudgable goes to at least $3/4$ for $n\to\infty.$
\end{rem}

\end{document}